\documentclass[12pt]{elsarticle}
\usepackage{amssymb}
\usepackage{amsthm}
\usepackage{enumerate}
\usepackage{amsmath,amssymb,amsfonts,amsthm,fancyhdr}
\usepackage{mathrsfs}
\usepackage[colorlinks,linkcolor=black,anchorcolor=black,citecolor=black,CJKbookmarks=True]{hyperref}
\usepackage{tikz}
\usepackage{mathabx}

\newdefinition{rem}{Remark}[section]
\newdefinition{theorem}{Theorem}[section]
\newdefinition{corollary}{Corollary}[section]
\newdefinition{definition}{Definition}[section]
\newdefinition{lemma}{Lemma}[section]
\newdefinition{prop}{Proposition}[section]
\numberwithin{equation}{section}

\begin{document}
\begin{frontmatter}
\title{Multifractal analysis of the growth rate of digits in Schneider's $p$-adic continued fraction dynamical system}
\author[a]{Kunkun Song}\ead{songkunkun@hunnu.edu.cn}
\author[b]{Wanlou Wu\corref{cor1}}\ead{wuwanlou@163.com}
\author[c]{Yueli Yu}\ead{yuyueli@whu.edu.cn}
\author[a]{Sainan Zeng}\ead{17376552448@163.com}
\address[a]{Key Laboratory of Computing and Stochastic Mathematics (Ministry of Education), School of Mathematics and Statistics, Hunan Normal University, Changsha, 410081, China}
\address[b]{School of Mathematics and Statistics, Jiangsu Normal University, Xuzhou, 221116, China}
\address[c]{School of Mathematics and Statistics, Wuhan University, Wuhan, 430072, China}
\cortext[cor1]{Corresponding author.} 
\begin{abstract}\par
   Let $\mathbb{Z}_p$ be the ring of $p$-adic integers and $a_n(x)$ be the $n$-th digit of Schneider's $p$-adic continued fraction of $x\in p\mathbb{Z}_p$. We study the growth rate of the digits $\{a_n(x)\}_{n\geq1}$ from the viewpoint of multifractal analysis. The Hausdorff dimension of the set \[E_{\sup}(\psi)=\Big\{x\in p\mathbb{Z}_p:\ \limsup\limits_{n\to\infty}\frac{a_n(x)}{\psi(n)}=1\Big\}\] is completely determined for any $\psi:\mathbb{N}\to\mathbb{R}^{+}$ satisfying $\psi(n)\to \infty$ as $n\to\infty$. As an application, we also calculate the Hausdorff dimension of the intersection sets \[E^{\sup}_{\inf}(\psi,\alpha_1,\alpha_2)=\left\{x\in p\mathbb{Z}_p:\liminf_{n\rightarrow\infty}\dfrac{a_n(x)}{\psi(n)}=\alpha_1,~\limsup_{n\rightarrow\infty}\dfrac{a_n(x)}{\psi(n)}=\alpha_2\right\}\] for the above function $\psi$ and $0\leq\alpha_1<\alpha_2\leq\infty$.
\end{abstract}
\end{frontmatter}

\section{Introduction}
   It is well known that regular continued fraction expansion can be induced by the Gauss map $T_G : [0,1)\rightarrow[0,1)$ defined as
\begin{equation*}
   T_G(0):=0,\ \ T_G(x):=1/x\ (\text{mod }1),\ \text{for}\ x\in(0,1).
\end{equation*}
   Each irrational number $x \in [0,1)$ admits a unique \emph{regular continued fraction expansion} of the form
\begin{equation}\label{cfe}
   x=\dfrac{1}{a_1(x) +\dfrac{1}{a_2(x)+\ddots+\dfrac{1}{a_n(x)+\ddots}}}=[a_1(x), a_2(x),\ldots, a_n(x),\ldots],
\end{equation}
   where $a_1(x)=\lfloor\frac{1}{x}\rfloor$ and $a_n(x)=a_1(T_G^{n-1}(x))\ (n\geq2)$ are called the \emph{partial quotients} of the regular continued fraction expansion of $x$.
   
   One of the major subjects in the study of regular continued fraction is concerned with the metric properties of the partial quotients. The Borel-Bernstein Theorem (see \cite{BF1912,BE1912}) asserts that for Lebesgue almost all $x \in [0,1)$, $a_n(x) \geq \psi(n)$ holds for infinitely many $n$'s or finitely many $n$'s depending on whether $\sum_{n \geq 1} 1/\psi(n)$ diverges or converges, which indicates that many sets of regular continued fraction for which their partial quotients obeying some restrictions are of null Lebesgue measure. Then it is natural to investigate the sizes of such null sets in the sense of Hausdorff dimension. The first published work was due to Jarn{\'i}k \cite{Jar28} who proved that the set of regular continued fraction whose partial quotients are bounded has full Hausdorff dimension. On the other hand, one can consider the sets of regular continued fraction whose partial quotients become large in some sense. Two typical sets are mainly concerned with the uniform and asymptotic behavior of the partial quotients, i.e., $$\big\{x\in[0,1):a_n(x)\geq\psi(n)\ \text{for all}\ n\in\mathbb{N}\big\},$$ $$\big\{x\in[0,1):a_n(x)\geq\psi(n)\ \text{for infinitely many}\ n\in\mathbb{N}\big\}.$$ The study of these two sets goes back to Good \cite{Good41} who proved that the set $\big\{x\in [0,1):a_n(x)\to\infty\ \text{as}\ n\to\infty\big\}$ is of Hausdorff dimension $1/2$. After that, many authors focus on the growth rate of partial quotients from various aspects in regular continued fraction. To list a few of them, we would like to mention the works of Hirst \cite{lesHir73}, {\L}uczak \cite{lesLuc97}, Feng, Wu, Liang and Tseng \cite{FWLT97}, Xu \cite{Xu08}, Liao and Rams \cite{LR161}.
   
   In particular, Wang and Wu \cite{WW08} completely determined the Hausdorff dimension of the set $\big\{x\in[0,1):a_n(x)\geq\psi(n)\ \text{for infinitely many}\ n\in\mathbb{N}\big\}$ for any $\psi:\mathbb{N}\to\mathbb{R}^{+}$ in regular continued fraction. Based on the main results in Wang and Wu \cite{WW08}, Fang, Ma and Song \cite{FMS21} recently obtained the Hausdorff dimensions of some exceptional sets of Borel-Bernstein theorem. Also, questions in the same flavor for continued fractions of Laurent series have been studied by  Hu, Wang, Wu and Yu \cite{HWWY08}, Hu, Hussain and Yu \cite{HHY21}. Motivated by the characterizations related to these two sets in regular continued fraction and continued fractions over the field of formal Laurent series, we will pay attention to the metric results in the setting of  Schneider's $p$-adic continued fraction.
   
   In the following we shall first introduce some basic facts about the field $\mathbb{Q}_p$ of $p$-adic numbers. Let $p$ be a prime number and $\mathbb{Q}$ be the field of rational numbers. It is well known that each non-zero rational number $r$ can be written as \[r=p^{v_p(r)}\cdot\frac{m}{n}\ \ \text{with}\ v_p(r), m, n\in\mathbb{Z}\ \text{and}\ (p,mn)=1,\] where $v_p(r)$ is called the $p$-adic valuation of $r$ and $(m,n)$ denotes the greatest common divisor of the two integers $m$ and $n$. We define the $p$-adic absolute value $|\cdot|_p$ on $\mathbb{Q}$ as \[|0|_p=0\ \text{and}\ |r|_p=p^{-v_p(r)}\ \text{for}\ r\neq0.\] Then $|\cdot|_p$ is a non-Archimedean absolute value (see \cite{Kob84}). The field $\mathbb{Q}_p$ of $p$-adic numbers is defined as the completion of $\mathbb{Q}$ with respect to $|\cdot|_p$, i.e., the smallest field containing $\mathbb{Q}$ in which all the Cauchy sequences are convergent. Thus, $\mathbb{Q}_p$ is obtained from $|\cdot|_p$ in the same way as the field $\mathbb{R}$ of real numbers is obtained from the Euclidean absolute value: as the completion of $\mathbb{Q}$. A typical element $x\in\mathbb{Q}_p$ is of the form \[x=\sum\limits_{i\geq v_p(x)}c_ip^{i}\ \text{with}\ c_i\in\mathbb{Z}/p\mathbb{Z}\ \text{ for all $i\geq v_p(x)$}\ \text{and}\ c_{v_p(x)}\neq0.\] The ring $\mathbb{Z}_p$ of $p$-adic integers is the compact subset of all $p$-adic numbers in the field $\mathbb{Q}_p$ with non-negative valuation. That is, \[\mathbb{Z}_p=\Big\{x\in\mathbb{Q}_p:\ |x|_p\leq1\Big\}=\left\{\sum_{i\geq0}c_ip^i:\ c_i\in\mathbb{Z}/p\mathbb{Z}\text{ for all $i\geq 0$}\right\}.\] Note that the set $x+\mathbb{Z}_p$ is a compact neighborhood of each $x\in\mathbb{Q}_p$, which implies that $\mathbb{Q}_p$ is a locally compact commutative group with respect to addition. Therefore, there exists a translation invariant Haar measure (see \cite{VVZ}) denoted by $\mu_p$ on $\mathbb{Q}_p$. We normalize the Haar measure $\mu_p$ by the equality $\mu_p(p\mathbb{Z}_p)=1$, where $p\mathbb{Z}_p$ is the unique maximal ideal of ring $\mathbb{Z}_p$.
   
   The Schneider's $p$-adic continued fractions map $T_p:\ p\mathbb{Z}_p\rightarrow p\mathbb{Z}_p$ is defined as
\begin{equation}\label{tp}
T_p(0)=0\ \text{and}\ T_p(x)=\frac{p^{a_1(x)}}{x}-b_1(x)\ \text{for}\ x\in p\mathbb{Z}_p\setminus\{0\},
\end{equation}
   where $a_1(x)=v_p(x)$ and $b_1(x)\in\{1,2,\ldots,p-1\}$ such that $p^{a_1(x)}/x=b_1(x)\ (\text{mod }p)$. Let
\begin{equation}\label{tan}
a_n=a_n(x)=a_1(T^{n-1}_p(x))\ \text{and}\ b_n=b_n(x)=b_1(T^{n-1}_p(x))\ \ \text{for all}\ n\geq2.
\end{equation}
   Then it follows from \eqref{tp} and \eqref{tan} that Schneider's $p$-adic continued fraction expansion (see \cite{S70,Van93}) of every element $x\in p\mathbb{Z}_p$ is of the form
\begin{align}\label{padic}
 \nonumber x&=\dfrac{p^{a_1(x)}}{b_1(x)+\dfrac{p^{a_2(x)}}{b_2(x)+\ddots+\dfrac{p^{a_n(x)}}{b_n(x)+T^{n}_p(x)}}}\\
  &:=[a_1(x),b_1(x);a_2(x),b_2(x);\ldots;a_n(x),b_n(x)+T^{n}_p(x)].
\end{align}
   Clearly, a finite Schneider's $p$-adic continued fraction expansion always represents a rational number, while the converse is not true. Bundschuh \cite{Bu} proved that if $x\in\mathbb{Q}$ has an infinite Schneider's $p$-adic continued fraction expansion, then this expansion has the same periodic tail, i.e., $a_n(x)=1,b_n(x)=p-1$ for every sufficiently large $n$. Hirst and Washington \cite{HW11} gave a combinatorial characterization of some non-terminating expansions of rational numbers. Pejkovi\'{c} \cite{P23} recently proved a criterion for determining whether the expansions of rational numbers terminate. For representations of quadratic irrationals, we would like to mention the theorem of Lagrange on regular continued fraction: let $x$ be an irrational positive real number, then its regular continued fraction expansion is eventually periodic if and only if $x$ is a quadratic irrational. However, Bundschuh \cite{Bu} showed that Lagrange's theorem on Schneider's $p$-adic continued fraction fails by some numerical computations. Later on, de Weger \cite{We88} gave a criterion for the non-periodicity of Schneider's $p$-adic continued fraction. Tilborghs \cite{Ti90} determined an algorithm to detect the non-periodicity of Schneider's $p$-adic continued fraction in finite steps. In other directions, Hirst and Washington \cite{HW11} proved that Schneider's $p$-adic continued fraction dynamical system $(p\mathbb{Z}_p,T_p,\mu_p)$ is ergodic. Bugeaud and Pejkovi\'{c} \cite{BP15} showed that Schneider's $p$-adic continued fractions are powerful to construct $p$-adic numbers with prescribed Diophantine properties. Hančl, Jaššová, Lertchoosakul and Nair \cite{HJLN13} obtained some results on the subsequence ergodic theory of the map $T_p$. Very recently, Haddley and Nair \cite{HN22} calculated the entropy of the dynamical system $(p\mathbb{Z}_p,T_p,\mu_p)$ and showed that it has a natural extension which is Bernoulli. For more details and background about $p$-adic numbers and Schneider's $p$-adic continued fraction, we refer to \cite{DKK,Kob84,R23,S70,Van93} and references therein.
   
   It is worth noting that Hu, Yu and Zhao \cite{HYZ18} proved that the digits $\{a_n(x)\}_{n\geq1}$ in the dynamical system $(p\mathbb{Z}_p,T_p,\mu_p)$ are independent and identically distributed with respect to $\mu_p$. Applying Borel-Cantelli Lemma, they also showed that for $\mu_p$-almost all $x \in p\mathbb{Z}_p$, $a_n(x) \geq \psi(n)$ holds for infinitely many $n$'s or finitely many $n$'s depending on whether $\sum_{n \geq 1} p^{-\psi(n)}$ diverges or converges. Thus, it is immediate that for $\mu_p$-almost all $x \in p\mathbb{Z}_p$,
\begin{equation}\label{limsup}
\limsup\limits_{n\to\infty}\frac{a_n(x)}{\log n}=\frac{1}{\log p}.
\end{equation}
   The purpose of this paper is to study the growth rate of the digits $\{a_n(x)\}_{n\geq1}$ in Schneider's $p$-adic continued fraction from the viewpoint of multifractal analysis. We first consider some exceptional sets of \eqref{limsup}, denoted by \[E_{\sup}(\psi)=\left\{x\in p\mathbb{Z}_p:\ \limsup\limits_{n\to\infty}\frac{a_n(x)}{\psi(n)}=1\right\},\] where $\psi:\mathbb{N}\rightarrow\mathbb{R}^+$ is a function satisfying $\psi(n)\to\infty$ as $n\to\infty$. In the sequel, we use the notation $\dim_{\rm H}$ to denote the Hausdorff dimension. Now we are in a position to state our main results.
\begin{theorem}\label{zdl}
   Let $\psi:\mathbb{N}\rightarrow\mathbb{R}^+$ be a function satisfying $\psi(n)\to\infty$ as $n\to\infty$. We have
\begin{enumerate}[(i)]
\item if $\psi(n)/n\to0$ as $n\to\infty$, then $\dim_{\rm H}E_{\sup}(\psi)=1$,
\item if $\psi(n)/n\to\alpha\ (0<\alpha<\infty)$ as $n\to\infty$, then $\dim_{\rm H}E_{\sup}(\psi)=s(\alpha)$,
where $s(\alpha)$ is the unique real solution of the equation
\begin{equation}\label{fc}
\sum\limits_{n\geq1}(p-1)p^{-(n+\alpha)s}=1,
\end{equation}
\item if $\psi(n)/n\to\infty$ as $n\to\infty$, then $\dim_{\rm H}E_{\sup}(\psi)=0$.
\end{enumerate}
\end{theorem}
\begin{rem}\label{slx}
   For any $0<\alpha<\infty$, we deduce from \eqref{fc} that  $p^{\alpha s}(p^s-1)=p-1$, which implies that the function $s:\mathbb{R}^+\rightarrow[0,1]$ is decreasing continuous satisfying
\[\lim\limits_{\alpha\to0}s(\alpha)=1\ \ \text{and}\ \ \lim\limits_{\alpha\to\infty}s(\alpha)=0.\]
\end{rem}
   From the proof Theorem \ref{zdl}, we could easily obtain the following result.
\begin{corollary}\label{lsc}
   Let $a_n(x)$ be the $n$-th digit of Schneider's $p$-adic continued fraction. Then \[\dim_{\rm H}\Big\{x \in p\mathbb{Z}_p:\ \limsup\limits_{n\to\infty}a_n(x)=\infty\Big\}=1.\]
\end{corollary}
   Replacing upper limit by lower limit or limit in the set $E_{\sup}(\psi)$, we also obtain the Hausdorff dimension of the sets \[E_{\inf}(\psi)=\left\{x \in p\mathbb{Z}_p: \liminf\limits_{n\to\infty}\frac{a_n(x)}{\psi(n)}=1\right\}\ \text{and}\ E(\psi)=\left\{x \in p\mathbb{Z}_p:\ \lim\limits_{n\to\infty}\frac{a_n(x)}{\psi(n)}=1\right\}.\]
\begin{theorem}\label{liminf}
   Let $\psi:\mathbb{N}\rightarrow\mathbb{R}^+$ be a function satisfying $\psi(n)\to\infty$ as $n\to\infty$. Then \[\dim_{\rm H}E_{\inf}(\psi)=\dim_{\rm H}E(\psi)=0.\]
\end{theorem}
   From the proof of Theorem \ref{liminf}, we immediately obtain the Hausdorff dimension of the uniform set defined by the digits $\{a_n(x)\}_{n\geq1}$ in Schneider's $p$-adic continued fraction.
\begin{corollary}\label{ry}
   Let $\psi:\mathbb{N}\rightarrow\mathbb{R}^+$ be a function satisfying $\psi(n)\to\infty$ as $n\to\infty$. Then \[\dim_{\rm H}\Big\{x\in p\mathbb{Z}_p:\ a_n(x)\geq\psi(n), \ \ \text{for all}\ n\in\mathbb{N}\Big\}=0.\]
\end{corollary}
   Based on Theorems \ref{zdl} and \ref{liminf}, we are also interested in the Hausdorff dimension of the intersection of the sets defined by the lower and upper limits. More precisely, for any $0\leq\alpha_1<\alpha_2\leq\infty$, we shall consider the Hausdorff dimension of the set \[E^{\sup}_{\inf}(\psi,\alpha_1,\alpha_2)=\left\{x\in p\mathbb{Z}_p:\liminf_{n\rightarrow\infty}\dfrac{a_n(x)}{\psi(n)}=\alpha_1,~\limsup_{n\rightarrow\infty}\dfrac{a_n(x)}{\psi(n)}=\alpha_2\right\}.\]
\begin{theorem}\label{liminfsup}
   Let $\psi:\mathbb{N}\rightarrow\mathbb{R}^+$ be a function satisfying $\psi(n)\to\infty$ as $n\to\infty$.
\begin{enumerate}[(i)]
\item For any $0<\alpha_1<\alpha_2\leq\infty$, we have $\dim_{\rm H}E^{\sup}_{\inf}(\psi,\alpha_1,\alpha_2)=0$.
\item For any $0=\alpha_1<\alpha_2<\infty$, we have that
\begin{enumerate}[(a)]
\item if $\psi(n)/n\to0$ as $n\to\infty$, then $\dim_{\rm H}E^{\sup}_{\inf}(\psi,\alpha_1,\alpha_2)=1$,
\item if $\psi(n)/n\to\alpha\ (0<\alpha<\infty)$ as $n\to\infty$, then $\dim_{\rm H}E^{\sup}_{\inf}(\psi,\alpha_1,\alpha_2)=s(\alpha\alpha_2)$,
where $s(\alpha\alpha_2)$ is the unique real solution of the equation
\[\sum\limits_{n\geq1}(p-1)p^{-(n+\alpha\alpha_2)s}=1,\]
\item if $\psi(n)/n\to\infty$ as $n\to\infty$, then $\dim_{\rm H}E^{\sup}_{\inf}(\psi,\alpha_1,\alpha_2)=0$.
\end{enumerate}
\end{enumerate}
\end{theorem}
\begin{rem}\label{sz}
   Replacing the level $1$ by other real numbers in the sets $E_{\inf}(\psi)$ and $E(\psi)$, we see that their Hausdorff dimensions are the same. While for the set $E_{\sup}(\psi)$, the dimensional result presents different phenomenon. To be precise, for any $0<\alpha_2<\infty\ (\alpha_2\neq1)$,
\begin{enumerate}[(i)]
\item if $\psi(n)/n\to0$ as $n\to\infty$, then
\[\dim_{\rm H}E_{\sup}(\psi)=\dim_{\rm H}\left\{x\in p\mathbb{Z}_p:\ \limsup\limits_{n\to\infty}\frac{a_n(x)}{\psi(n)}=\alpha_2\right\}=1,\]
\item if $\psi(n)/n\to\alpha\ (0<\alpha<\infty)$ as $n\to\infty$, then
\[\dim_{\rm H}E_{\sup}(\psi)\neq\dim_{\rm H}\left\{x\in p\mathbb{Z}_p:\ \limsup\limits_{n\to\infty}\frac{a_n(x)}{\psi(n)}=\alpha_2\right\}=\dim_{\rm H}E^{\sup}_{\inf}(\psi,0,\alpha_2),\]
\item if $\psi(n)/n\to\infty$ as $n\to\infty$, then
\[\dim_{\rm H}E_{\sup}(\psi)=\dim_{\rm H}\left\{x\in p\mathbb{Z}_p:\ \limsup\limits_{n\to\infty}\frac{a_n(x)}{\psi(n)}=\alpha_2\right\}=0.\]
\end{enumerate}
\end{rem}

   In the end, let us remark that the convergence exponent (see P\'{o}lya and Szeg\H{o} \cite[p.26]{PS72}) of the sequence of the digits $\{a_n(x)\}_{n \geq 1}$ in Schneider's $p$-adic continued fraction, given by
\begin{equation*}
\tau(x):= \inf\Big\{s \geq 0: \sum\limits_{n \geq 1} a^{-s}_n(x)<\infty\Big\},
\end{equation*}
   reflects how fast the growth rate of the digits $\{a_n(x)\}_{n\geq1}$ tending to infinity. Let \[T(\alpha)=\big\{x\in p\mathbb{Z}_p:\ \tau(x)=\alpha\big\},\quad(0\leq\alpha\leq\infty).\] Then we shall calculate the multifractal spectrum of $\tau(x)$, i.e, the dimensional function $\alpha\mapsto \dim_{\rm H}T(\alpha)$.
\begin{theorem}\label{cv}
For any $0\leq\alpha\leq\infty$, we have
\begin{equation*}
\dim_{\rm H}T(\alpha)=
\begin{cases}
1,\ \ \ \ \ \ \ \ \alpha=\infty,\cr
0,\ \ \ 0\leq\alpha<\infty.
\end{cases}
\end{equation*}
\end{theorem}

   The paper is organized as follows. In Section $2$, some properties of Schneider's $p$-adic continued fraction and useful lemmas for calculating the Hausdorff dimension in the $p$-adic field are provided. The proofs of the main results are given in Section $3$.
   
\section{Preliminaries}
\subsection{Elementary properties of Schneider's p-adic continued fraction}
   For any $n\geq1$ and $(a_n,b_n)\in\mathbb{N}\times\{1,2,\ldots,p-1\}$, we call
\begin{equation*}
I_{n}(a_1,b_1;\ldots;a_n,b_n): =\big\{x\in p\mathbb{Z}_p:\ a_i(x)=a_i, b_i(x)=b_i,\ 1\leq i\leq n\big\}
\end{equation*}
   a cylinder of order $n$ of $p$-adic continued fraction. Denote the $n$-th convergent of $x$ by
\begin{equation}\label{pq}
 \frac{P_n(x)}{Q_n(x)}:=[a_1(x),b_1(x);\ldots;a_n(x),b_n(x)]=\dfrac{p^{a_1(x)}}{b_1(x) +\dfrac{p^{a_2(x)}}{b_2(x)+\ddots+\dfrac{p^{a_n(x)}}{b_n(x)}}},
\end{equation}
   where $P_n(x)$ and $Q_n(x)$ are $p$-adic integers. Let \[P_0(x)=0,\ Q_0(x)=1,\ P_1(x)=p^{a_1(x)},\ Q_1(x)=b_1(x).\] Then $P_n(x)$ and $Q_n(x)$ satisfy the following recursive formula:
\begin{equation}\label{pq2}
\begin{cases}
P_n(x)=b_n(x)P_{n-1}(x)+p^{a_n(x)}P_{n-2}(x)\ \ (n\geq2),\cr
Q_n(x)=b_n(x)Q_{n-1}(x)+p^{a_n(x)}Q_{n-2}(x)\ (n\geq2).
\end{cases}
\end{equation}
   The following lemma is concerned with the fundamental properties of $P_n(x)$ and $Q_n(x)$.
\begin{lemma}[{\cite{HW11,Van93}}]\label{lem1}
For any $n\geq1$, we have
\begin{enumerate}[(i)]
\item $\big(p,Q_n(x)\big)=1$,\ $\big(P_n(x),Q_n(x)\big)=1$,
\item $P_{n-1}(x)Q_n(x)-P_n(x)Q_{n-1}(x)=(-1)^{n}\cdot p^{a_1(x)+\cdots+a_n(x)}$.
\end{enumerate}
\end{lemma}

   Next we shall characterize the structure of the cylinder of order $n$. Note that for any $x\in I_{n}(a_1,b_1;\cdots;a_n,b_n)$, we deduce from \eqref{padic}, \eqref{pq} and \eqref{pq2} that \[x=\frac{P_n(x)+T^{n}_p(x)P_{n-1}(x)}{Q_n(x)+T^{n}_p(x)Q_{n-1}(x)},\ \text{with}\ T^{n}_p(x)\in p\mathbb{Z}_p,\] which, in combination with Lemma \ref{lem1}, implies that
\begin{align*} x-\frac{P_n(x)}{Q_n(x)}&=\frac{P_n(x)+T^{n}_p(x)P_{n-1}(x)}{Q_n(x)+T^{n}_p(x)Q_{n-1}(x)}-\frac{P_n(x)}{Q_n(x)}\\
&=\frac{(-1)^{n}p^{a_1(x)+\cdots+a_n(x)}T^{n}_p(x)}{Q_n(x)(Q_n(x)+T^{n}_p(x)Q_{n-1}(x))}
\in p^{1+a_1(x)+\cdots+a_n(x)}\mathbb{Z}_p.
\end{align*}
   In other words, we have
\begin{equation}\label{zjbs}
I_{n}(a_1,b_1;\ldots;a_n,b_n)=\frac{P_n(x)}{Q_n(x)}+p^{1+a_1(x)+\cdots+a_n(x)}\mathbb{Z}_p.
\end{equation}
   It is a fact that the field $\mathbb{Q}_p$ of $p$-adic numbers is non-Archimedean with respect to $|\cdot|_p$ (see \cite{Kob84}), then every point that is contained in a ball is a centre of that ball. This shows that
\begin{prop}[\cite{HW11}]\label{zjkh}
   For any $x\in I_{n}(a_1,b_1;\ldots;a_n,b_n)$, the cylinder can be viewed as a ball with centre $x$ and radius $p^{-(1+a_1(x)+\cdots+a_n(x))}$, i.e., \[I_{n}(a_1,b_1;\ldots;a_n,b_n)=\Big\{y\in p\mathbb{Z}_p:\ |y-x|_p\leq p^{-\big(1+a_1(x)+\cdots+a_n(x)\big)}\Big\}.\]
\end{prop}

   We now turn to analyse the distributions of the cylinder of order $n$ in the sense of the normalized Haar measure $\mu_p$ on $p\mathbb{Z}_p$, which satisfies (see\ {\cite[p. 399]{Fal90}})
\begin{equation}\label{cd}
\mu_p(pa+p^{k}\mathbb{Z}_p)=p^{1-k},\ (a\in\mathbb{Z}_p,\ k\geq1).
\end{equation}
   It follows from \eqref{zjbs} and \eqref{cd} that
\begin{equation}\label{zjcd}
\mu_p\big(I_{n}(a_1,b_1;\ldots;a_n,b_n)\big)=p^{-\big(a_1(x)+\cdots+a_n(x)\big)}.
\end{equation}
\begin{prop}\label{dg}
For any $k\geq1$,
\[\mu_p\big(\{x\in p\mathbb{Z}_p:\ a_1(x)=k\}\big)=(p-1)p^{-k}.\]
\end{prop}
\begin{proof}
   Using the non-Archimedean property of the field $\mathbb{Q}_p$, we know that any two distinct of the balls $I_1(a_1,b_1)$ for $(a_1,b_1)\in\mathbb{N}\times\{1,2,\ldots,p-1\}$ are disjoint. Then by \eqref{zjcd}, \[\mu_p\big(\{x\in p\mathbb{Z}_p:\ a_1(x)=k\}\big)=\sum\limits_{1\leq b_1\leq p-1}\mu_p\big(I_1(k,b_1)\big)=(p-1)p^{-k}.\]
\end{proof}
\subsection{Hausdorff dimension in the $p$-adic field}
   Let us first recall the definition of Hausdorff dimension in $p\mathbb{Z}_p$. Let $U$ be any non-empty subset of $p\mathbb{Z}_p$. We denote the diameter of $U$ by \[{\rm diam}(U):=\sup\big\{|x-y|_p:\ x,y\in U\big\}.\] For any $E\subseteq p\mathbb{Z}_p$ and $\delta>0$, the sequence of sets $\{U_i\}_{i\geq1}$ is called a $\delta$-cover of $E$ if $E\subseteq\bigcup_{i\geq1}U_i$ and ${\rm diam}\ U_i\leq\delta$ for all $i\in\mathbb{N}$. Let $s$ be a non-negative real number and let \[\mathcal{H}^s_{\delta}(E)=\inf\left\{\sum_{i\geq1}\big({\rm diam}(U_i)\big)^s:\{U_i\}_{i\geq1} \text{ is a $\delta$-cover of } E\right\}.\] It is easy to see that the quantity $\mathcal{H}^s_{\delta}(E)$ is increasing as $\delta$ decreases, then the $s$-dimensional Hausdorff measure of $E$ is defined as \[\mathcal{H}^s(E):=\lim_{\delta\rightarrow0}\mathcal{H}^s_{\delta}(E)=\sup_{\delta>0}\mathcal{H}^s_{\delta}(E).\] Based on the value of $\mathcal{H}^s(E)$, one can define the Hausdorff dimension of $E$ (see \cite{Fal90}) as \[{\rm dim_H}E:=\inf\Big\{s\geq0:\mathcal{H}^s(E)=0\Big\}:=\sup\Big\{s\geq0:\mathcal{H}^s(E)=\infty\Big\}.\] More details of Hausdorff dimension in the $p$-adic field can be found in \cite{Ab1,Ab2,DDY,LQ19,Me}.
   
   In the following we shall collect and establish some useful lemmas for calculating the Hausdorff dimension of certain sets in Schneider's p-adic continued fraction. The first lemma is to deal with the Jarn\'{\i}k's type set whose digits $\{a_n(x)\}_{n\geq1}$ are bounded for $x\in p\mathbb{Z}_p$. For any positive integer $M\geq2$, let $E_M$ be the set of points in $p\mathbb{Z}_p$ whose digits $\{a_n(x)\}_{n\geq1}$ in Schneider's p-adic continued fraction do not exceed $M$. That is, \[E_{M}=\Big\{x\in p\mathbb{Z}_p:\ 1\leq a_n(x)\leq M, \ \text{for all}\  n\in\mathbb{N}\Big\}.\]
\begin{lemma}\label{yj}
   Let $s_M=\dim_{\rm H}E_{M}$. Then $\lim_{M\to\infty}s_M=1$.
\end{lemma}
\begin{proof}
   Note that the set $E_M$ can be viewed as the attractor (also called self-similar set) of the iterated function system $\{f_{ij}(x)\}_{1\leq i\leq M,1\leq j\leq p-1}$ given by \[f_{ij}(x)=\frac{p^{i}}{x+j},\ x\in p\mathbb{Z}_p.\] For any $x,y\in p\mathbb{Z}_p$, the contracting ratio of $f_{ij}$ under the $p$-adic absolute value $|\cdot|_p$ is \[\left|\frac{f_{ij}(x)-f_{ij}(y)}{x-y}\right|_{p}=\left|\frac{p^i}{(x+j)(y+j)}\right|_{p}=\frac{|p^i|_p}{|x+j|_p\cdot|y+j|_p}=p^{-i}.\] On the other hand, the iterated function system $\{f_{ij}(x)\}_{1\leq i\leq M,1\leq j\leq p-1}$ satisfy the open set condition. In other words, there exists a bounded open set $p\mathbb{Z}_p$ such that \[p\mathbb{Z}_p\supseteq\bigcup\limits_{i=1}^{M}\bigcup\limits_{j=1}^{p-1}f_{ij}(p\mathbb{Z}_p)\] with the union disjoint. Together with the above facts, we conclude from \cite[Theorem 9.3]{Fal90} that the  number $s_M$ is the unique real solution of the self-similar equation \[\sum\limits_{j=1}^{p-1}\sum\limits_{i=1}^{M}(p^{-i})^{s_M}=1.\] The proof is completed by a simple computation.
\end{proof}

   The second lemma is due to Hu, Yu and Zhao \cite{HYZ18}, which concerns the Hausdorff dimension of the set defined by the arithmetic mean of the digits $\{a_n(x)\}_{n\geq1}$.
\begin{lemma}[{\cite[Lemma 4.5]{HYZ18}}]\label{pj}
   Let $\{a_n(x)\}_{n\geq1}$ be the digits in Schneider's $p$-adic continued fraction. Then \[\dim_{\rm H}\Big\{x\in p\mathbb{Z}_p:\ \lim\limits_{n\to\infty}\frac{a_1(x)+\cdots+a_n(x)}{n}=\infty\Big\}=0.\]
\end{lemma}
   From Lemma \ref{pj}, we could obtain the Hausdorff dimension of Good's type set in Schneider's $p$-adic continued fraction directly.
\begin{lemma}\label{qyww}
   Let $a_n(x)$ be the $n$-th digit of Schneider's $p$-adic continued fraction. Then \[\dim_{\rm H}\Big\{x\in p\mathbb{Z}_p:\ a_n(x)\to\infty\ \text{as}\ n\to\infty\Big\}=0.\]
\end{lemma}

\section{Proofs of main results}
   In this section, the proofs of main results are divided into three parts. We first prove Theorem \ref{zdl} and Corollary \ref{lsc}, and then prove Theorems \ref{liminf} and \ref{cv}, and at last present the proof of Theorem \ref{liminfsup}.
\subsection{Proofs of Theorem \ref{zdl} and Corollary \ref{lsc}}
   Recall that \[E_{\sup}(\psi)=\left\{x\in p\mathbb{Z}_p:\ \limsup\limits_{n\to\infty}\frac{a_n(x)}{\psi(n)}=1\right\}.\] We divide the proof of Theorem \ref{zdl} into three cases.
\subsubsection{\textbf{The case} $\psi(n)/n\to0$ as $n\to\infty$.}
   In this case, our strategy is to construct a suitable Cantor subset $E_{M}(\psi)$ of $E_{\sup}(\psi)$ and then establish a connection between $E_{M}(\psi)$ and $E_M$ by means of a function satisfying a H\"{o}lder condition. Consider $\{m_k\}_{k\geq 1}$ with $m_k = 2^k$. For any $M\geq2$, define
\begin{eqnarray}\label{emp}
  \nonumber E_M(\psi)=\Big\{x\in p\mathbb{Z}_p:\ a_{m_k}(x)=\lfloor\psi(m_k)\rfloor+1\ \text{for all} \ \ k\geq 1 \
   \text{and} \\
   1\leq a_i(x)\leq M\ \text{for}\ i\neq m_k\ \text{for any}\ k\geq1\Big\}.
\end{eqnarray}
   Here and in the sequel, the notation $\lfloor\cdot\rfloor$ means the integer part of a real number. From the definition of $E_M(\psi)$, we obtain the following lemma directly.
\begin{lemma}\label{phibh}
   For any $M\geq2$, $E_M(\psi)\subseteq E_{\sup}(\psi)$.
\end{lemma}
   Next, we shall use the following result to estimate the Hausdorff dimension of $E_M(\psi)$.
\begin{lemma}[{\cite[Proposition 2.3]{Fal90}}]\label{Fal-Lip}
   Let $E\subseteq p\mathbb{Z}_p$ and suppose that $f:E\to p\mathbb{Z}_p$ satisfies a H{\"o}lder condition \[|f(x)-f(x)|_p\le c|x-y|_p^{\alpha}\quad(x,y\in E)\] for constants $c>0$ and $\alpha>0$. Then $\dim_{\rm H}f(E)\le 1/\alpha \dim_{\rm H} E$.
\end{lemma}

\begin{lemma}\label{phimw} 
   Let $E_M(\psi)$ be defined as above. Then we have \[\dim_{\rm H}E_M(\psi)\ge\dim_{\rm H}E_M.\]
\end{lemma}

\begin{proof}
   It suffices to prove that for any $\varepsilon>0$, \[\dim_{\rm H}E_M(\psi)\ge\frac{1}{1+\varepsilon}\dim_{\rm H}E_M.\] Recall that \[m_k=2^k,\quad \lim\limits_{n\to\infty}\psi(n)/n=0.\] Thus there exists an integer $k_0\ge1$ such that
\begin{equation}\label{ge1+epsilon}
   1-\frac{\lfloor\psi(m_1)\rfloor+\cdots+\lfloor\psi(m_k)\rfloor+M+k+\lfloor\psi(m_k)\rfloor}{m_k}\ge \frac{1}{1+\varepsilon}
\end{equation}
   for all $k\ge k_0$. For any $x\in E_M(\psi)$, we define a map $f: E_M(\psi)\to E_M$ by \[f(x)=[a_1(x),b_1(x);\ldots;a_{m_i-1}(x),b_{m_i-1}(x);a_{m_i+1}(x),b_{m_i+1}(x);\ldots].\] That is, if we delete  all the digits $a_{m_i}(x),b_{m_i}(x)$ from the $p$-adic continued fraction of $x$, then we get  the $p$-adic continued fraction digits of $f(x)$. Thus, the map $f$ is surjective. Let \[n_0=m_{k_0}\ \text{and}\ \delta=\min{\rm diam}(I_{n_0}(a_1,b_1;\ldots;a_{n_0},b_{n_0})),\] where the minimum is taken over all $(a_1,b_1;\ldots;a_{n_0},b_{n_0})$ satisfying $1\le a_i\le M$ for $i\not\in\{m_k\}_{k\geq1}$, $ a_{i}=\lfloor\psi(i)\rfloor+1$ for $i\in \{m_k\}_{k\geq1}$ and $ b_i\in \{1,\ldots,p-1\}$ for all $1\le i\le n_0$. Then for any $x,y\in E_M(\psi)$ with $|x-y|_p<\delta$, we have $a_i(x)=a_i(y)$ and $b_i(x)=b_i(y)$ for all $1\le i\le n_0$. Let $n>n_0$ be the smallest integer such that
\begin{equation}\label{df-of-n}
  a_n(x)\ne a_n(y)\ \text{and}\ b_n(x)\ne b_n(y).
\end{equation}
   Then there exists a unique integer $k\ge1$ such that
\begin{equation}\label{n-and-mk}
   m_k\le n <m_{k+1}.
\end{equation}
   Note that $a_i(x)=a_i(y)$ and $b_i(x)=b_i(y)$ for all $1\le i\le n-1$, then we have \[f(x),f(y)\in I_{l}(a_1(x),b_1(x);\ldots;a_{m_i-1}(x),b_{m_i-1}(x);a_{m_i+1}(x),b_{m_i+1}(x);\ldots;a_{n-1}(x),b_{n-1}(x)),\] where $ l=n-1-k$ if $ n>m_k$ and $ l=n-k$ if $ n=m_k$. By Proposition \ref{zjkh}, it follows that
\begin{equation}\label{f-ge}
   |f(x)-f(y)|_p\le p^{-1-\big(\sum\limits_{i=1}^{ n-1}a_i(x)-\sum\limits_{i=1}^{k}a_{m_i}(x)\big)}.
\end{equation}
   On the other hand, the two cylinders \[I_{n}(a_1(x),b_1(x);\ldots;a_n(x),b_n(x)),\quad I_{n}(a_1(y),b_1(y);\ldots;a_n(y),b_n(y))\] are disjoint from \eqref{df-of-n}. By Proposition \ref{zjkh}, it follows that
\begin{equation}\label{x-y-le}
   |x-y|_p\geq p^{-\big(1+\sum\limits_{i=1}^{ n}a_i(x)\big)}.
\end{equation}
   By \eqref{n-and-mk} and the definition of $E_M(\psi)$, we have $a_{n}(x)\le M+\lfloor \psi(m_k)\rfloor$ and
\[
\begin{aligned}
  \frac{1+\sum\limits_{i=1}^{ n-1}a_i(x)-\sum\limits_{i=1}^{k}a_{m_i}(x)}{1+\sum\limits_{i=1}^{ n}a_i(x)}
  &=1-\frac{a_n(x)+\sum\limits_{i=1}^{k}(1+\lfloor\psi(m_i)\rfloor)}{1+\sum\limits_{i=1}^{ n}a_i(x)}\\
&\ge 1-\frac{M+k+\lfloor\psi(m_k)\rfloor+\sum\limits_{i=1}^{k}\lfloor\psi(m_i)\rfloor}{m_k}\\
&\ge \frac{1}{1+\varepsilon},
\end{aligned}
\]
   where the last inequality follows by \eqref{ge1+epsilon}. We deduce from \eqref{f-ge} and \eqref{x-y-le} that \[|f(x)-f(y)|_p\le |x-y|_p^{1/(1+\varepsilon)}\] holds for any $x,y\in E_M(\psi)$ with $|x-y|_p<\delta$. Let $c=\max\{1,\delta^{-1/(1+\varepsilon)}\}$. Then \[|f(x)-f(y)|_p\le c|x-y|_p^{1/(1+\varepsilon)}\] holds for any $x,y\in E_M(\psi)$. By Lemma \ref{Fal-Lip}, the conclusion follows.
\end{proof}
   Thus, the conclusion of Theorem \ref{zdl} follows in this case by Lemmas \ref{phibh}, \ref{phimw} and \ref{yj}.

\subsubsection{\textbf{The case} $\psi(n)/n\to\alpha\ (0<\alpha<\infty)$ as $n\to\infty$.}
   In this case, it is easy to see that  the set $E_{\sup}(\psi)$ equals to \[F(\alpha):=\left\{x\in p\mathbb{Z}_p:\ \limsup\limits_{n\to\infty}\frac{a_n(x)}{n}=\alpha\right\}.\] Hence, in order to get $\dim_{\rm H}E_{\sup}(\psi)$, it is equivalent to calculating $\dim_{\rm H}F(\alpha)$.\\
\textbf{Upper bound:} The strategy here is to construct a larger suitable set containing $F(\alpha)$ by using its definition. For any $0<\varepsilon<\alpha$, let
\begin{eqnarray*}
\mathcal{C}_n=\big\{(a_1,b_1;\ldots;a_n,b_n):\ (a_k,b_k)\in\mathbb{N}\times\{1,\ldots,p-1\}\ \text{for any}\ 1\leq k\leq n-1,\\
\text{and}\  a_n\geq n(\alpha-\varepsilon), 1\leq b_n\leq p-1\big\}.
\end{eqnarray*}
   For any $x\in F(\alpha)$, then it follows from the definition of upper limit that there are infinitely many $n$'s such that $a_n(x)\geq n(\alpha-\varepsilon)$. Thus, we have
\begin{align}\label{xs}
\nonumber F(\alpha)&\subseteq\Big\{x\in p\mathbb{Z}_p:\ a_n(x)\geq n(\alpha-\varepsilon)\ \text{for infinitely many}\ n\in\mathbb{N}\Big\}\\
\nonumber&=\bigcap_{N\geq1}\bigcup_{n\geq N}\big\{x\in p\mathbb{Z}_p:\ a_n(x)\geq n(\alpha-\varepsilon)\big\}\\
&=\bigcap_{N\geq1}\bigcup_{n\geq N}\bigcup_{(a_1,b_1;\cdots;a_n,b_n)\in\mathcal{C}_n}I_n(a_1,b_1;\ldots;a_n,b_n).
\end{align}
   Note that for any $t>s(\alpha-\varepsilon)$, we deduce from \eqref{fc} and Remark \ref{slx} that
\begin{equation}\label{bds}
\sum\limits_{k\geq1}(p-1)p^{-(k+\alpha-\varepsilon)t}<1.
\end{equation}
   It follows from \eqref{xs}, \eqref{bds} and Proposition \ref{zjkh} that
\begin{align*}
\mathcal{H}^{t}(F(\alpha))&\leq\liminf\limits_{N\to\infty}\sum\limits_{n\geq N}\sum\limits_{(a_1,b_1;\ldots;a_n,b_n)\in\mathcal{C}_n}\big({\rm diam}(I_n(a_1,b_1;\ldots;a_n,b_n))\big)^t\\
&\leq\liminf\limits_{N\to\infty}\sum\limits_{n\geq N}\sum\limits_{(a_1,b_1;\ldots;a_n,b_n)\in\mathcal{C}_n}
p^{-(1+a_1+\cdots+a_{n-1}+a_n)t}\\
&\leq\liminf\limits_{N\to\infty}\sum\limits_{n\geq N}(p-1)^{n}\cdot\sum\limits_{a_1,\ldots,a_{n-1}\geq1}
p^{-\big(a_1+\cdots+a_{n-1}+(n-1)(\alpha-\varepsilon)\big)t}\\
&=(p-1)\cdot\liminf\limits_{N\to\infty}\sum\limits_{n\geq N}\left(\sum\limits_{k\geq1}(p-1)p^{-(k+\alpha-\varepsilon)t}\right)^{n-1}=0.
\end{align*}
   This shows that $\dim_{\rm H}F(\alpha)\leq t$ for any $t>s(\alpha-\varepsilon)$. We deduce from Remark \ref{slx} that \[\dim_{\rm H}F(\alpha)\leq\lim\limits_{\varepsilon\to0^+}s(\alpha-\varepsilon)=s(\alpha).\]
\textbf{Lower bound:} To bound $\dim_{\rm H}F(\alpha)$ from below, we shall construct a suitable Cantor subset of $F(\alpha)$. Here we follow the constructions used in \cite[Lemma 4.7]{HYZ18}. Let $\{n_k\}_{k\geq0}$ be a sequence of positive integers satisfying
$n_0=1$ and $n_{k+1}\geq(k+2)n_k,\ (k\geq0)$. For any positive integer $M\geq2$, let
\begin{eqnarray}\label{zzj}
  \nonumber F(n_k,\alpha,M)=\Big\{x\in p\mathbb{Z}_p:\ a_{n_k}(x)=\lfloor\alpha n_k\rfloor+1\ \text{for all} \ \ k\geq 0 \
   \text{and} \\
   1\leq a_i(x)\leq M\ \text{for}\ i\neq n_k\ \text{for any}\ k\geq0\Big\}.
\end{eqnarray}
   Then it is clear that \[F(\alpha)\supseteq F(n_k,\alpha,M).\] From the result in \cite[Lemma 4.7]{HYZ18}, we know that $\dim_{\rm H}F(n_k,\alpha,M)$ is not less than the unique real solution of the equation \[\sum\limits_{1\leq k\leq M}(p-1)p^{-(k+\alpha)s}=1.\] By letting $M\to\infty$, we compare the above equation  with \eqref{fc} and conclude that
\begin{equation}\label{ess}
   \dim_{\rm H}F(\alpha)\geq\dim_{\rm H}F(n_k,\alpha,M)\geq s(\alpha).
\end{equation}

\subsubsection{\textbf{The case} $\psi(n)/n\to\infty$ as $n\to\infty$.}
   In this case, for any $\alpha>0$, \[E_{\sup}(\psi)\subseteq\Big\{x\in p\mathbb{Z}_p:\ a_n(x)\geq n\alpha \ \text{for infinitely many}\ n\in\mathbb{N}\Big\}.\] In view of the process of calculating $\dim_{\rm H}F(\alpha)$, we could obtain $\dim_{\rm H}E_{\sup}(\psi)\leq s(\alpha)$. By letting $\alpha\to\infty$, we deduce from Remark \ref{slx} that \[\dim_{\rm H}E_{\sup}(\psi)\leq\lim\limits_{\alpha\to\infty}s(\alpha)=0.\]
   
\textbf{Proof of Corollary \ref{lsc}:} It is worth pointing out that  the set $F(n_k,\alpha,M)$ in \eqref{zzj} satisfies
\[F(n_k,\alpha,M) \subseteq \Big\{x \in p\mathbb{Z}_p:\ \limsup\limits_{n\to\infty}a_n(x)=\infty\Big\}.\] Then it follows that
\[\dim_{\rm H}\Big\{x \in p\mathbb{Z}_p:\ \limsup\limits_{n\to\infty}a_n(x)=\infty\Big\}\geq s(\alpha).\]
By letting $\alpha\to0$, we obtain the desired result from Remark \ref{slx}.

\subsection{Proofs of Theorems \ref{liminf} and \ref{cv}}
   By the definitions of limit and lower limit, we could show that
\begin{align*}
E_{\inf}(\psi),\ E(\psi)&\subseteq\Big\{x\in p\mathbb{Z}_p:\ a_n(x)\geq\frac{\psi(n)}{2}\ \ \text{for all}\ n\in\mathbb{N}\ \text{large enough}\Big\}\\
&\subseteq\Big\{x\in p\mathbb{Z}_p:\ a_n(x)\to\infty\ \text{as}\ n\to\infty\Big\}.
\end{align*}
   It follows from Lemma \ref{qyww} that \[\dim_{\rm H}E_{\inf}(\psi)=\dim_{\rm H}E(\psi)=0.\] Now it turns to present the proof of Theorem \ref{cv}. Recall that \[T(\alpha)=\big\{x\in p\mathbb{Z}_p:\ \tau(x)=\alpha\big\},\ (0\leq\alpha\leq\infty).\] Applying Birkhoff Ergodic Theorem to the digits $\{a_n(x)\}_{n\geq1}$ in dynamical system $(p\mathbb{Z}_p,T_p,\mu_p)$, we deduce from Proposition \ref{dg} that for $t>0$ and $\mu_p$-almost all $x \in p\mathbb{Z}_p$,
\begin{align*}
\lim\limits_{n\to\infty}\frac{a^{-t}_1(x)+\cdots+a^{-t}_n(x)}{n}
&=\lim\limits_{n\to\infty}\frac{a^{-t}_1(x)+\cdots+a^{-t}_1(T^{n-1}_p(x))}{n}\\
&=\int_{p\mathbb{Z}_p}a^{-t}_1(x)d\mu_p(x)=\sum\limits_{k\geq1}k^{-t}\cdot\mu_p\{x\in p\mathbb{Z}_p:\ a_1(x)=k\}\\
&=(p-1)\sum\limits_{k\geq1}k^{-t}\cdot p^{-k}<\infty.
\end{align*}
   This shows that $\tau(x)=\infty$ for $\mu_p$-almost all $x \in p\mathbb{Z}_p$, and thus $\dim_{\rm H}T(\infty)=1$. For the case $0\leq\alpha<\infty$, we deduce from the definition of $\tau(x)$ that for any $\varepsilon>0$, \[\sum_{n\geq1}a^{-(\alpha+\varepsilon)}_n(x)<\infty\ \text{and thus}\ a_n(x)\to\infty\ \text{as}\ n\to\infty.\] This implies that \[T(\alpha)\subseteq\Big\{x\in p\mathbb{Z}_p:\ a_n(x)\to\infty\ \text{as}\ n\to\infty\Big\}.\] By Lemma \ref{qyww}, we have \[\dim_{\rm H}T(\alpha)=0.\]
   
\subsection{Proof of Theorem \ref{liminfsup}}
   Recall that \[E^{\sup}_{\inf}(\psi,\alpha_1,\alpha_2)=\left\{x\in p\mathbb{Z}_p:\liminf_{n\rightarrow\infty}\dfrac{a_n(x)}{\psi(n)}=
\alpha_1,~\limsup_{n\rightarrow\infty}\dfrac{a_n(x)}{\psi(n)}=\alpha_2\right\}.\]
\textbf{For the case $0<\alpha_1<\alpha_2\leq\infty$}, it is easy to see that
\begin{align*}
E^{\sup}_{\inf}(\psi,\alpha_1,\alpha_2)&\subseteq\Big\{x\in p\mathbb{Z}_p:\ \liminf_{n\rightarrow\infty}\dfrac{a_n(x)}{\psi(n)}=
\alpha_1\Big\}\\
&\subseteq\Big\{x\in p\mathbb{Z}_p:\ a_n(x)\to\infty\ \text{as}\ n\to\infty\Big\}.
\end{align*}
   By Lemma \ref{qyww}, we have \[\dim_{\rm H}E^{\sup}_{\inf}(\psi,\alpha_1,\alpha_2)=0.\]
\textbf{For the case $0=\alpha_1<\alpha_2<\infty$}, we divide the proof into three cases.
\begin{enumerate}[(a)]
\item If $\psi(n)/n\to0$ as $n\to\infty$, from the set $E_{M}(\psi)$ constructed in \eqref{emp}, we have
\[E_{M}(\psi)\subseteq E^{\sup}_{\inf}(\psi,\alpha_1,\alpha_2).\]
By letting $M\to\infty$, we deduce from Lemmas \ref{phimw} and \ref{yj} that
\[\dim_{\rm H}E^{\sup}_{\inf}(\psi,\alpha_1,\alpha_2)=1.\]
\item If $\psi(n)/n\to\alpha\ (0<\alpha<\infty)$ as $n\to\infty$, then it is clear that
\[E^{\sup}_{\inf}(\psi,\alpha_1,\alpha_2)\subseteq\left\{x\in p\mathbb{Z}_p:\ \limsup\limits_{n\to\infty}\frac{a_n(x)}{n}=\alpha\alpha_2\right\}:=F(\alpha\alpha_2).\]
   Using the same method in the proof of the second part of Theorem \ref{zdl}, we can obtain $\dim_{\rm H}E^{\sup}_{\inf}(\psi,\alpha_1,\alpha_2)$. In fact, for the upper bound of $\dim_{\rm H}E^{\sup}_{\inf}(\psi,\alpha_1,\alpha_2)$, it follows from the definition of upper limit that for any $0<\varepsilon<\alpha\alpha_2$, \[F(\alpha\alpha_2)\subseteq\Big\{x\in p\mathbb{Z}_p:\ a_n(x)\geq n(\alpha\alpha_2-\varepsilon)\ \text{for infinitely many}\ n\in\mathbb{N}\Big\}.\] Then from the estimation of the upper bound of $\dim_{\rm H}F(\alpha)$, we have \[\dim_{\rm H}E^{\sup}_{\inf}(\psi,\alpha_1,\alpha_2)\leq\dim_{\rm H}F(\alpha\alpha_2)\leq\lim\limits_{\varepsilon\to0^+}s(\alpha\alpha_2-\varepsilon)=s(\alpha\alpha_2).\] In the following we shall estimate the lower bound of $\dim_{\rm H}E^{\sup}_{\inf}(\psi,\alpha_1,\alpha_2)$. Let
\begin{eqnarray*}
F(n_k,\alpha\alpha_2,M)=\Big\{x\in p\mathbb{Z}_p:\ a_{n_k}(x)=\lfloor\alpha\alpha_2 n_k\rfloor+1\ \text{for all} \ \ k\geq 0 \
   \text{and} \\
   1\leq a_i(x)\leq M\ \text{for}\ i\neq n_k\ \text{for any}\ k\geq0\Big\}.
\end{eqnarray*}
   Here the set $F(n_k,\alpha\alpha_2,M)$ comes from the set $F(n_k,\alpha,M)$ marked in \eqref{zzj} by changing $\alpha$ to $\alpha\alpha_2$. Then it is easy to verify that \[E^{\sup}_{\inf}(\psi,\alpha_1,\alpha_2)\supseteq F(n_k,\alpha\alpha_2,M).\] By letting $M\to\infty$, we deduce from \eqref{ess} that \[\dim_{\rm H}E^{\sup}_{\inf}(\psi,\alpha_1,\alpha_2)\geq\dim_{\rm H}F(n_k,\alpha\alpha_2,M)\geq s(\alpha\alpha_2).\]
\item If $\psi(n)/n\to\infty$ as $n\to\infty$, then it follows that
\begin{align*}
E^{\sup}_{\inf}(\psi,\alpha_1,\alpha_2)&\subseteq\Big\{x\in p\mathbb{Z}_p:\ \limsup_{n\rightarrow\infty}\dfrac{a_n(x)}{\psi(n)}=
\alpha_2\Big\}\\
&\subseteq\Big\{x\in p\mathbb{Z}_p:\ \limsup_{n\rightarrow\infty}\dfrac{a_n(x)}{n}=
\infty\Big\}.
\end{align*}
   This means that for any $\alpha>0$, \[E^{\sup}_{\inf}(\psi,\alpha_1,\alpha_2)\subseteq\Big\{x\in p\mathbb{Z}_p:\ a_n(x)\geq n\alpha\ \text{for infinitely many}\ n\in\mathbb{N}\Big\}.\] By letting $\alpha\to\infty$, we conclude from Remark \ref{slx} that \[\dim_{\rm H}E^{\sup}_{\inf}(\psi,\alpha_1,\alpha_2)\leq\lim\limits_{\alpha\to\infty}s(\alpha)=0.\]
\end{enumerate}
{\bf Acknowledgement:} The authors would like to thank Professor Lingmin Liao for his invaluable comments in an early version of the paper. The research is supported by Natural Science Foundation of Jiangsu Province
(No.BK20201025) and National Natural Science Foundation of China (Nos.12001245,12201207,12371072).\\
{\bf Conflict of Interest:} The authors declared that they have no conflict of interest.


\begin{thebibliography}{99}
\bibitem{Ab1} A.G. Abercrombie, {\it The Hausdorff dimension of some exceptional sets of p-adic integer matrices}, J. Number Theory 53 (2) (1995), 311--341.

\bibitem{Ab2} A.G. Abercrombie, {\it Badly approximable p-adic integers}, Proc. Indian Acad. Sci. Math. Sci. 105 (2) (1995), 123--134.

\bibitem{BF1912} F. Bernstein, {\it Über eine Anwendung der Mengenlehre auf ein der Theorie der säkularen Störungen herrührendes Problem}, Math. Ann. 71 (1912), 417--439.

\bibitem{BE1912} E. Borel, {\it Sur un problème de probabilités relatif aux fractions continues}, Math. Ann. 72 (1912) 578--584.

\bibitem{Bu}P. Bundschuh, {\it p-adische Kettenbrüche und Irrationalität p-adischer Zahlen}, Elem. Math. 32 (2) (1977), 36--40.

\bibitem{BP15} Y. Bugeaud, T. Pejkovi\'{c}, {\it Quadratic approximation in $\mathbb{Q}_p$}, Int. J. Number Theory 11 (1) (2015), 193--209.

\bibitem{DDY}H. Dickinson, M.M. Dodson, J. Yuan, {\it Hausdorff dimension and p-adic Diophantine approximation}, Indag. Math. 10 (3) (1999), 337--347.

\bibitem{DKK}B. Dragovich, A.Y. Khrennikov, S.V. Kozyrev, I.V. Volovich, E.I. Zelenov, {\it $p$-adic mathematical physics: the first 30 years}, $p$-Adic Numbers Ultrametric Anal. Appl. 9 (2) (2017), 87--121.

\bibitem{Fal90} K.J. Falconer, {\it Fractal Geometry: Mathematical Foundations and Applications}, Wiley, New York, 2003.

\bibitem{FMS21}L.L. Fang, J.H. Ma, K.K. Song, {\it Some exceptional sets of Borel-Bernstein theorem in continued fractions}, Ramanujan J. 56 (3) (2021), 891--909.

\bibitem{FWLT97} D.J. Feng, J. Wu, J.C. Liang and S. Tseng,
{\it Appendix to the paper by T. {\L}uczak--A simple proof of the lower bound:
``On the fractional dimension of sets of continued fractions''},
Mathematika, 44(1)(1997), 54--55.


\bibitem{Good41} I.J. Good, {\it The fractional dimensional theory of continued fractions}, Proc. Cambridge Philos. Soc. 37 (1941), 199--228.

\bibitem{HN22} A. Haddley, R. Nair, {\it On Schneider's continued fraction map on a complete non-Archimedean field}, Arnold Math. J. 8 (1) (2022), 19--38.

\bibitem{HJLN13} J. Han{\v c}l,  A. Ja{\v s}{\v s}ov{\' a}, P. Lertchoosakul, R. Nair, {\it On the metric theory of p-adic continued fractions}, Indag. Math. (N.S.) 24 (1) (2013), 42--56.

\bibitem{HW11} J. Hirst, L.C. Washington, {\it p-adic continued fractions}, Ramanujan J. 25 (3) (2011), 389--403.

\bibitem{lesHir73} K.E. Hirst, {\it Continued fractions with sequences of partial quotients}, Proc. Amer. Math. Soc. 38 (1973), 221--227.

\bibitem{HYZ18} H. Hu, Y.L. Yu, Y.F. Zhao, {\it On the digits of Schneider's p-adic continued fractions}, J. Number Theory 187 (2018), 372--390.

\bibitem{HHY21}H. Hu, M. Hussain, Y.L. Yu, {\it Metrical properties for continued fractions of formal Laurent series}, Finite Fields Appl. 73 (2021), Paper No. 101850.

\bibitem{HWWY08} X.H. Hu, B.W. Wang, J. Wu, Y.L. Yu, {\it Cantor sets determined by partial quotients of continued fractions of Laurent series}, Finite Fields Appl. 14 (2) (2008), 417--437.

\bibitem{Jar28} V. Jarn\'{\i}k, {\it Zur metrischen Theorie der diopahantischen Approximationen}, Proc. Mat. Fyz. 36 (1928), 91--106.

\bibitem{Kob84} N. Koblitz, {\it p-adic Numbers, p-adic Analysis, and Zeta-functions}, 2nd ed., Graduate
Texts in Mathematics, 58, Springer-Verlag, New York, 1984.

\bibitem{LQ19} Y. Li, H. Qiu, {\it Fractal sets in the field of p-adic analogue of the complex numbers}, Fractals, 27 (4) (2019), Paper No. 1950053.

\bibitem{LR161}L.M. Liao, M. Rams, {\it Subexponentially increasing sums of partial quotients in continued fraction expansions}, Math. Proc. Camb. Phil. Soc. 160 (3) (2016), 401--412.

\bibitem{lesLuc97} T. {\L}uczak, {\it On the fractional dimension of sets of continued fractions}, Mathematika 44 (1997), 50--53.

\bibitem{Me}J.V. Melni{\v c}uk, {\it Hausdorff dimension in Diophantine approximation of p-adic numbers}, Ukra\"{\i}n. Mat. Zh., 32 (1980), 118--124.

\bibitem{P23} T. Pejkovi\'{c}, {\it Schneider's p-adic continued fractions}, Acta Math. Hungar. 169 (1) (2023), 191--215.

\bibitem{PS72} G. P\'{o}lya, G. Szeg\H{o}, {\it Problems and Theorems in Analysis Vol. I}, Springer-Verlag, Berlin, 1972.

\bibitem{R23} G. Romeo, {\it Continued fractions in the field of p-adic numbers}, Bull. Amer. Math. Soc.(N.S.) 61 (2024), no.2, 343--371.

\bibitem{S70} Th. Schneider, Über p-adische Kettenbrüche. In: Symposia Mathematica, Vol. IV, INDAM, Rome, 1968/69,  Academic Press, London, 1970, pp. 181--189.

\bibitem{Ti90} F. Tilborghs, {\it Periodic p-adic continued fractions}, Simon Stevin 64 (3) (1990), 383--390.


\bibitem{Van93} A.J. Van der Poorten, {\it Schneider's continued fraction, in: Number Theory with an Emphasis on the Markoff Spectrum, in: Lecture Notes in Pure and Applied Mathematics}, vol. 147, Marcel Dekker, New York, 1993, pp. 271--281.

\bibitem{VVZ} V.S. Vladimirov, I.V. Volovich, E.I. Zelenov, {\it p-adic analysis and mathematical
physics, Series on Soviet and East European Mathematics, Vol. 1}, World Scientific Publishing Co. Inc., River Edge, 1994.

\bibitem{WW08} B.W. Wang, J. Wu, {\it Hausdorff dimension of certain sets arising in continued fraction expansions}, Adv. Math. 218 (2008), 1319--1339.

\bibitem{We88} B.M.M. de Weger, {\it Periodicity of p-adic continued fractions}, Elemente der Math. 43 (1988),
112--116.

\bibitem{Xu08}J. Xu, {\it On sums of patial quotients in continued fraction expansions}, Nonlinearity 21 (2008), no. 9, 2113--2120.


\end{thebibliography}
\end{document}